\documentclass{article}
\usepackage[utf8]{inputenc}
\usepackage{ dsfont }
\usepackage{amsfonts}
\usepackage{amsmath}
\usepackage{commath}
\usepackage{amsthm}
\usepackage{ amssymb }
\usepackage{enumerate} 
\usepackage{mathtools}
\usepackage{verbatim}
\usepackage{hyperref}
\usepackage[pagewise]{lineno}
\usepackage{graphicx}
\usepackage{float}

\newtheorem{theorem}{Theorem}

\newtheorem{lemma}{Lemma}

\title{On ${L}(2,1)$-labelings of some products of oriented cycles}
\author{
  Lucas Colucci$^{1,2}$\thanks{Research of the author was partially supported by the National Research, Development and Innovation, NKFIH grant K 116769.}\\
  \texttt{lucas.colucci.souza@gmail.com}
  \and
  Ervin Gy\H ori$^{1,2}$\thanks{Research of the author was partially supported by the National Research, Development and Innovation, NKFIH grants K 116769 and SNN 117879.}\\
  \texttt{gyori.ervin@renyi.hu}
}

\date{  $^1$Alfréd Rényi Institute of Mathematics, Hungarian Academy of Sciences, Re\'altanoda~u.~13--15, 1053 Budapest, Hungary\\%
        $^2$Central European University, Department of Mathematics and its Applications, N\'ador~u.~9, 1051 Budapest, Hungary\\[2ex]%
        \today}

\begin{document}

\maketitle

\begin{abstract}

   We refine two results of Jiang, Shao and Vesel on the $L(2,1)$-labeling number $\lambda$ of the Cartesian and the strong product of two oriented cycles. For the Cartesian product, we compute the exact value of $\lambda(\overrightarrow{C_m} \square \overrightarrow{C_n})$ for $m$, $n \geq 40$; in the case of strong product, we either compute the exact value or establish a gap of size one for $\lambda(\overrightarrow{C_m} \boxtimes \overrightarrow{C_n})$ for $m$, $n \geq 48$.

\end{abstract}

\section{Introduction}

A \emph{$L(p,q)$-labeling}, or \emph{$L(p,q)$-coloring}, of a graph $G$ is a function $f:V(G)\rightarrow\{0,\dots,k\}$ such that $|f(u)-f(v)|\geq p$, if $e=uv \in E(G)$; and $|f(u)-f(v)|\geq q$, if there is a path of length two in $G$ joining $u$ and $v$. To take into account the number of colors used, we say that $f$ is a $k$-$L(p,q)$-labeling of $G$ (note that, for historical reasons, the colorings are assumed to start with the label 0). The minimum value of $k$ such that  $G$ admits a $k$-$L(p,q)$-labeling is denoted by $\lambda_{p,q}(G)$, and it is called the \emph{$L(p,q)$-labeling number} of $G$. 

\medskip

The particular case of $L(p,q)$-labelings that attracted the most attention is $p = 2$ and $q = 1$, the $L(2,1)$-labeling. It was introduced by Yeh \cite{yeh1990labeling}, and it traces back to the frequency assignment problem of wireless networks introduced by Hale \cite{hale1980frequency}. In this case, we write $\lambda(G)$ instead of $\lambda_{2,1}(G)$ for short.

\medskip

The definitions above can be extended to oriented graphs (a directed graph whose underlying graph is simple), namely: if $G$ is an oriented graph, a $L(p,q)$-labeling of $G$ is a function $f:V(G)\rightarrow\{0,\dots,k\}$ such that $|f(u)-f(v)|\geq p$, if $uv \in E(G)$; and $|f(u)-f(v)|\geq q$, if there is a \emph{directed} path of length two in $G$ joining $u$ and $v$. The corresponding $L(p,q)$-labeling number is again denoted by $\lambda_{p,q}(G)$ (in some papers, the notation $\overrightarrow{\lambda}_{p,q}(G)$ is used instead). The $L(2,1)$-labelings of oriented graphs were first studied by Chang and Liaw \cite{chang20032}, and the $L(p,q)$-labeling problem has been extensively studied since then in both undirected and directed versions. We refer the interested reader to the excellent surveys of Calamoneri \cite{calamoneri2011h} and Yeh \cite{yeh2006survey}.

\medskip

In this paper, we study the $L(2,1)$-labeling number of the Cartesian and the strong product of two oriented cycles, improving results of Jiang, Shao and Vesel \cite{shao20182}. We use the notation $\overrightarrow{C_n}$ to represent the oriented cycle on $n$ vertices, i.e., the digraph such that $V(\overrightarrow{C_n}) = \{1,2,\dots,n\}$ and $E(\overrightarrow{C_n})=\{(1,2),(2,3),\dots,(n-1,n),(n,1)\}$, $n \geq 3$. In the case of Cartesian product, we compute the exact value of $\lambda(\overrightarrow{C_m} \square \overrightarrow{C_n})$ for $m$, $n \geq 40$; in the case of strong product, we either compute the exact value or establish a gap of size one for $\lambda(\overrightarrow{C_m} \boxtimes \overrightarrow{C_n})$ for $m$, $n \geq 48$.

\section{Cartesian product}

The \emph{Cartesian product} of two graphs (resp. digraphs) $G$ and $H$ is the graph (resp. digraph) $G \square H$ such that $V(G \square H) = V(G) \times V(H)$, and where there is an edge joining $(a,x)$ and $(b,y)$ if $ab \in E(G)$ and $x=y$, or if $a=b$ and $xy \in E(H)$  (resp. there is an edge pointing from $(a,x)$ to $(b,y)$ if $ab \in E(G)$ and $x=y$, or if $a=b$ and $xy \in E(H)$).

\begin{figure}[H]
    \centering
    
    \includegraphics[width=8cm, height=6cm]{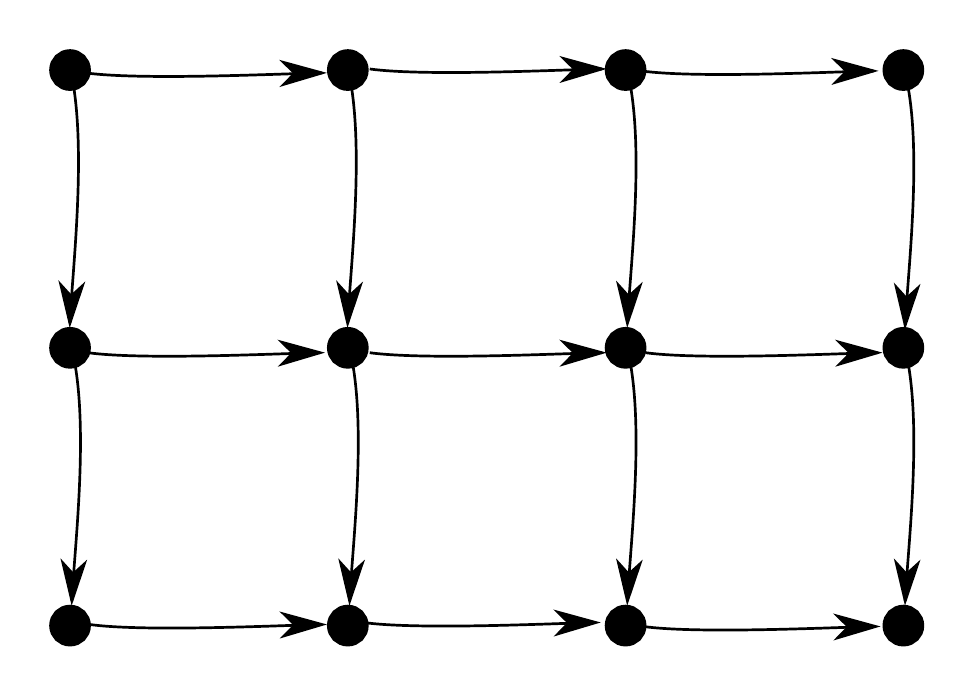}

    \caption{The Cartesian product of $\protect\overrightarrow{P_3}$ and $\protect\overrightarrow{P_4}$}
    
    \label{cart}
    
\end{figure}

Let $S(m,n) = \{am+bn: \text{$a$, $b \geq 0$ integers not both zero}\}$. A classical result of Sylvester \cite{sylvester1884mathematical} states that $t \in S(m,n)$ for all integers $t \geq (m-1)(n-1)$ that are divisible by $\gcd(m,n)$, the greatest common divisor of $m$ and $n$.

\medskip

In \cite{shao20182}, Jiang, Shao and Vesel prove the following theorem:

\begin{theorem}[\cite{shao20182}]
For all $m$, $n \in S(5,11)$, $4 \leq \lambda(\overrightarrow{C_m} \square \overrightarrow{C_n}) \leq 5$. In particular, the result holds for every $m$, $n \geq 40$.
\end{theorem}{}

Our result in this section determines the exact value of $\lambda$ in the range above. We start with a lemma which is a slightly stronger version of Lemma 5 from \cite{shao20182} that can be obtained with the same proof, which we include here for the sake of completeness.

\begin{lemma}\label{diagonalcart}
For every $m$, $n \geq 3$ and every 4-$L(2,1)$-labeling $f$ of $\overrightarrow{C_m} \square \overrightarrow{C_n}$, the following periodicity condition holds:

\begin{equation}\label{4period}
    f(i,j)=f(i+1 \text{ mod $m$},j-1 \text { mod $n$}) \text{ for all } i \in [m], j \in [n].
\end{equation}
\end{lemma}

\begin{proof}
Let $f$ be a $4$-$L(2,1)$-labeling of $G=\overrightarrow{C_m} \square \overrightarrow{C_n}$. We write $f(i,j)$ for the value of $f$ at the vertex $(i,j)$.
By the symmetry of the graph, it is enough to prove that $f(2,2)=f(1,3)$. Suppose, for the sake of contradiction, that this is not the case.

Consider the $\overrightarrow{P_2} \square \overrightarrow{P_2}$ subgraph spanned by $\{(1,2),(1,3),(2,2),(2,3)\}$. Every vertex is at distance at most 2 from each other, expect for the pair $(2,2), (1,3)$. This implies, together with our assumption that $f(2,2) \neq f(1,3)$, that every vertex of this subgraph gets a distinct color. It is clear that the color 2 cannot be used in any vertex $v$ of this subgraph, since otherwise the two neighbours of $v$ must receive colors 0 and 4, and there is no color left for the fourth vertex. Thus, the colors used on the vertices of this subgraph are exactly 0, 1, 3 and 4, in some order.

Using the fact that $4-f$ is a $4$-$L(2,1)$-labeling of a graph whenever $f$ is, we may assume without loss of generality that $f(2,2) \in \{0,1\}$. This implies that $f(1,3) \in \{0,1\}$ and $\{f(1,2),f(2,3)\} = \{3,4\}$. If $f(2,3)=3$, there is no color for $(3,3)$, since its within distance two from $(2,2)$ and $(1,3)$, colored with 0 and 1, and it is neighbor of a vertex of color 3. If $f(1,2)=3$, the same argument applies for the vertex $(1,1)$.
\end{proof}

We call labelings with the property of Lemma \ref{diagonalcart} \emph{diagonal}.

\medskip

The following lemma from \cite{shao20182} combined with the result of Sylvester will also help us:

\begin{lemma}\label{multiple}(Lemmas 2 and 3 in \cite{shao20182})
Let $m$, $n$, $p \geq 3$ and $t$, $k \geq 1$ be integers. If $\lambda(\overrightarrow{C_m} \square \overrightarrow{C_n}) \leq k$ and $\lambda(\overrightarrow{C_p} \square \overrightarrow{C_n}) \leq k$, then $\lambda(\overrightarrow{C_{m}}_{+tp} \square \overrightarrow{C_n}) \leq k$.

In particular, if $m$ and $n$ are such that $\lambda(\overrightarrow{C_m} \square \overrightarrow{C_n}) \leq k$, $\lambda(\overrightarrow{C_m} \square \overrightarrow{C_m}) \leq k$ and $\lambda(\overrightarrow{C_n} \square \overrightarrow{C_n}) \leq k$, then $\lambda(\overrightarrow{C_a} \square \overrightarrow{C_b}) \leq k$ for all $a$, $b \in S(m,n)$, and hence for all $a$, $b \geq (m-1)(n-1)$ divisible by $\gcd(m,n)$.
\end{lemma}

\begin{theorem}\label{large}
Let $m, n \geq 40$. Then:

$$
{\lambda}(\overrightarrow{C_m} \square \overrightarrow{C_n}) = 
\begin{cases}
4, \text{ if } \gcd(m,n) \geq 3;\\
5, \text{ otherwise. }\\
\end{cases}
$$
\end{theorem}

\begin{proof}

\medskip

For $m, n \geq 3$, let $G$ denote the graph $\overrightarrow{C_m} \square \overrightarrow{C_n}$, i.e., $V(G) = [m] \times [n]$ and the directed edges of $G$ point from $(i,j)$ to $(i+1 \text{ mod $m$}, j)$ and to $(i,j+1 \text{ mod $n$})$, for every $i \in [m], j \in [n]$. For a labeling $f$, we write $f(i,j)$ instead of $f((i,j))$ for short.

\medskip

Let $d = \gcd(m,n)$ and assume first that $d \notin \{1, 2\}$. According to Lemma \ref{multiple}, it is enough to prove that $\lambda(\overrightarrow{C_d} \square \overrightarrow{C_d}) = 4$. Any 4-$L(2,1)$-labeling $f$ of $\overrightarrow{C_d}$ can be extended to a 4-$L(2,1)$-labeling $f'$ of $\lambda(\overrightarrow{C}_d \square \overrightarrow{C_d})$ by setting $f'(i,j) = f(i+j \text{ mod $d$})$. It suffices to show, then, that $\lambda(\overrightarrow{C_d}) = 4$.

\medskip

If $d \equiv 0$ (mod 3), then we can label $\overrightarrow{C_d}$ with $d/3$ blocks 024. If $d \equiv 1$ (mod 3), we label $\overrightarrow{C_d}$ with $(d-4)/3$  consecutive blocks 024 and then one block 0314. Finally, if $d \equiv 2$ (mod 3), then we label $\overrightarrow{C_d}$ with $(d-2)/3$ consecutive blocks 024 and then a block 13.

\bigskip 

On the other hand, assume for the sake of contradiction that $d \in \{1, 2\}$ and there is a 4-$L(2,1)$-labeling $f$ of $\overrightarrow{C_m} \square \overrightarrow{C_n}$. In particular, $m \neq n$, so let us assume that $m > n$. 

\medskip

It is easy to check that, if $m \geq n+3$, $f$ induces a valid 4-$L(2,1)$-labeling of $\overrightarrow{C_m}_{-n} \square \overrightarrow{C_n}$. In fact,  let $g(i,j) = f(i,j)$ for all $1 \leq i \leq m-n$ and $1 \leq j \leq n$. We claim that $g$ is a 4-$L(2,1)$-labeling of $\overrightarrow{C_m}_{-n} \square \overrightarrow{C_n}$, which, in particular, satisfies (\ref{4period}) as well.

\medskip

Indeed, all we have to check is that the following conditions hold for $g$, since the other restrictions are inherited by $f$: $|g(m-n-1,j)-g(1,j)| \geq 1$, $|g(m-n,j)-g(1,j)| \geq 2$, $|g(m-n,j)-g(2,j)| \geq 1$, $|g(m-n,j)-g(1,j+1 \text{ mod $n$})| \geq 2$, for every $j \in [n]$. All these conditions follow from $g(m-n-1,j) = f(m-n-1, j) = f(m-1, j+n \text{ mod $n$}) = f(m-1, j)$ and $g(m-n,j) = f(m-n, j) = f(m, j+n \text{ mod $n$}) = f(m, j)$, which result from the application of (\ref{4period}) $n$ times, together with the fact that $f$ is a $L(2,1)$-labeling of $\overrightarrow{C_m} \square \overrightarrow{C_n}$.

\medskip

Applying this argument consecutively, using the fact that $d = \gcd(m,n)$ and by the symmetry of the factors of the product, we conclude that $f$ induces a 4-$L(2,1)$-labeling $c$ of either $\overrightarrow{C_k}_{+1} \square \overrightarrow{C_k}$ or $\overrightarrow{C_k}_{+2} \square \overrightarrow{C_k}$, for some $k \geq 3$. This is a contradiction, since in this case we would have, by Lemma \ref{diagonalcart}, $c(1,1) = c(2,k) = \dots = c(k+1,1)$ and $(k+1,1)$ and $(1,1)$ are joined by and edge or by a directed path of length two, respectively. 
\medskip

\end{proof}

\section{Strong product}

The \emph{strong product} of two graphs (resp. digraphs) $G$ and $H$ is the graph (resp. digraph) $G \boxtimes H$ such that $V(G \boxtimes H) = V(G) \times V(H)$, and where there is an edge joining $(a,x)$ and $(b,y)$ if either $ab \in E(G)$ and $x=y$, or if $a=b$ and $xy \in E(H)$, or if $ab \in E(G)$ and $xy \in E(H)$.  (resp. either there is an edge pointing from $(a,x)$ to $(b,y)$ if $ab \in E(G)$ and $x=y$, or if $a=b$ and $xy \in E(H)$, or $ab \in E(G)$ and $xy \in E(H)$).

\begin{figure}[H]
    \centering
    
    \includegraphics[width=8cm, height=6cm]{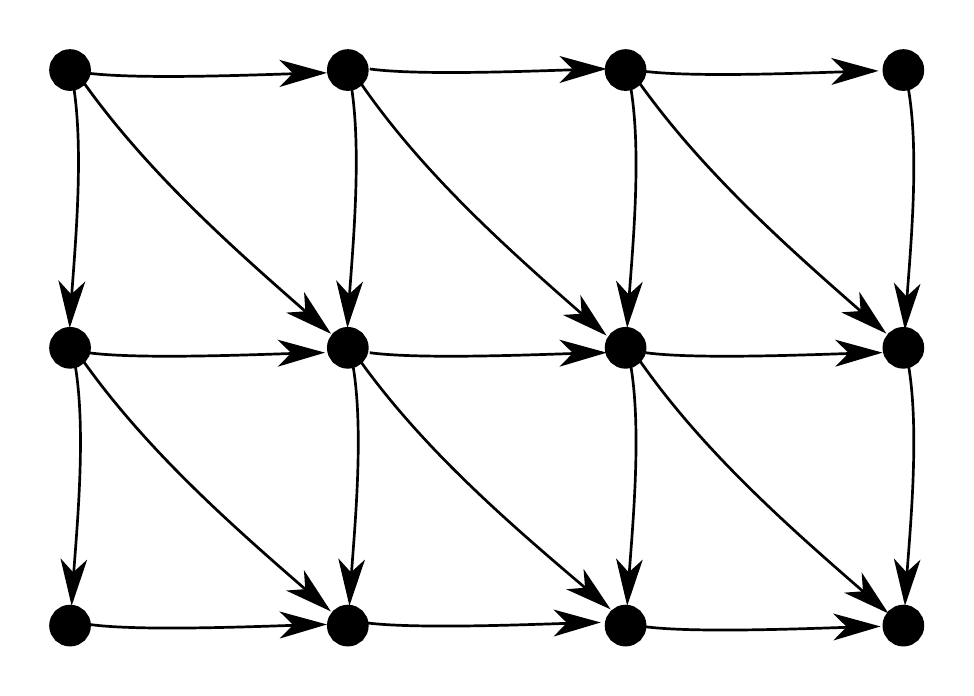}

    \caption{The strong product of $\protect\overrightarrow{P_3}$ and $\protect\overrightarrow{P_4}$}
    
    \label{strong}
    
\end{figure}

In the same paper, Jiang, Shao and Vesel prove the following result for the strong product of two directed cycles:

\begin{theorem}[\cite{shao20182}]\label{strongold}
If $m$, $n \geq 48$, then $6 \leq \lambda(\overrightarrow{C_m} \boxtimes \overrightarrow{C_n}) \leq 8$.
\end{theorem}

In this section, we refine this theorem in the following way:

\begin{theorem}\label{strong}
If $m$, $n \geq 48$, then

\[ \lambda(\overrightarrow{C_m} \boxtimes \overrightarrow{C_n}) = 
  \begin{cases*}
    6, & if $m \equiv n \equiv 0 \pmod 7$; \\
    7 \text{ or } 8, & otherwise.\\
  \end{cases*}\]  

\end{theorem}

The key lemma of in the proof of Theorem \ref{strong} is analogous to Lemma \ref{diagonalcart}:

\begin{lemma}\label{strongperiod}
Let $m$, $n \geq 4$ be integers. Any 6-$L(2,1)$-labeling $f$ of $\overrightarrow{C_m} \boxtimes \overrightarrow{C_n}$ is diagonal, i.e., the following condition holds:

\begin{equation}\label{6period}
    f(i,j)=f(i+1 \text{ mod $m$},j-1 \text { mod $n$}) \text{ for all } i \in [m], j \in [n].
\end{equation}
\end{lemma}

\begin{proof}[Proof of Lemma \ref{strongperiod}]

Let $G$ be the graph $\overrightarrow{C_m} \boxtimes \overrightarrow{C_n}$. For every vertex $(i,j)$ of $G$, there is a $\overrightarrow{P_4} \boxtimes \overrightarrow{P_4}$ subgraph of $G$ as in Figure \ref{figstrong} such that $v_{23}$ is the vertex $(i,j)$. It suffices, then, to show that $f(v_{23}) = f(v_{32})$ for every 6-$L(2,1)$-labeling of the graph in Figure \ref{figstrong}.

\medskip

Let $f$ be a 6-$L(2,1)$-labeling of $G$. By the fact that $6-f$ is also a 6-$L(2,1)$-labeling of $G$, we may assume that $f(v_{32}) \in \{0,1,2,3\}$. We will divide the rest of the proof in cases according to the value of $f(v_{32})$. In each case, we will assume that $f(v_{23}) \neq f(v_{32})$ and reach a contradiction by applying the rules of $L(2,1)$-labeling and finding a vertex for which there is no available color. We will use the notation $v!$ to mean that there is no color available for the vertex $v$, and the notation $\{u,v\} \in S!$ to mean that $u$ and $v$ cannot be colored using the colors in $S$, where $S$ is the set of possible colors for $u$ and $v$ based on the colors of the previous vertices and the rules of $L(2,1)$-labeling.  For instance, if $u$ and $v$ are joined by an edge, then $\{u,v\} \in \{0,1\}!$.

\subsubsection*{Case 1: $f(v_{32}) = 3$}

\begin{itemize}

    \item $f(v_{23}) \in \{5,6\} \Rightarrow  \{v_{22},v_{33}\} \in \{0,1\}!$.
    
    \item $f(v_{23}) = 4 \Rightarrow f(v_{22}) = 6$ and $f(v_{33}) \in \{0,1\}$, or $f(v_{22}) \in \{0,1\}$ and $f(v_{33}) = 6$. In the first case, $f(v_{43}) = 5 \Rightarrow f(v_{44}) = 2 \Rightarrow f(v_{33}) =0 \Rightarrow v_{34}!$; in the second, $f(v_{21}) = 5 \Rightarrow f(v_{11}) = 2 \Rightarrow f(v_{22}) =0 \Rightarrow v_{12}!$.
    
    \item $f(v_{23}) \in \{0,1\} \Rightarrow \{v_{12},v_{23}\} \in \{5,6\}!$.
    
    \item $f(v_{23}) = 2 \Rightarrow f(v_{22}) = 0$ and $f(v_{33}) \in \{5,6\}$, or $f(v_{33}) = 0$ and $f(v_{22}) \in \{5,6\}$. In the first case,  $\Rightarrow f(v_{43}) = 1 \Rightarrow f(v_{44}) = 4 \Rightarrow f(v_{33}) = 6 \Rightarrow v_{34}!$; in the second, $\Rightarrow f(v_{21}) = 1 \Rightarrow f(v_{11}) = 4 \Rightarrow f(v_{22}) = 6 \Rightarrow v_{12}!$.
    
\end{itemize}

\subsubsection*{Case 2: $f(v_{32}) = 0$}

\begin{itemize}

    \item $f(v_{23}) = 3 \Rightarrow \{v_{22},v_{33}\} \in \{5,6\}! $.
    
    \item $f(v_{23}) = 5 \Rightarrow \{v_{22},v_{33}\} \in \{2,3\}! $.
    
    \item $f(v_{23}) = 2 \Rightarrow f(v_{22}) = 6$ and $f(v_{33}) = 4$, or $f(v_{33}) = 6$ and $f(v_{22}) = 4$. In the first case, $ v_{43}! $; in the second, $v_{21}!$.
    
    \item $f(v_{23}) = 4 \Rightarrow f(v_{22}) = 6$ and $f(v_{33}) = 2$, or $f(v_{22}) = 2$ and $f(v_{33}) = 6$. In the first case, $v_{34}! $; in the second, $v_{12}!$.
    
    \item $f(v_{23}) = 6 \Rightarrow f(v_{22}) = 2$ and $f(v_{33}) = 4$, or $f(v_{22}) = 4$ and $f(v_{33}) = 2$. In the first case, $ v_{43}! $; in the second, $v_{21}!$.
    
    \item $f(v_{23}) = 1 \Rightarrow  f(v_{22}) = 3$ and $f(v_{33}) = 5$, or $f(v_{22}) = 5$ and $f(v_{33}) = 3$,  or $f(v_{22}) = 6$ and $f(v_{33}) = 4$, or $f(v_{22}) = 4$ and $f(v_{33}) = 6$,  or $f(v_{22}) = 6$ and $f(v_{33}) = 3$, or $f(v_{22}) = 3$ and $f(v_{33}) = 6$. In the first and the third cases, $v_{34}!$; in the second and the fourth, $v_{12}!$; in the fifth, $f(v_{34})=5 \Rightarrow v_{44}!$; in the sixth, $f(v_{12})=5 \Rightarrow v_{11}!$.
    
\end{itemize}

\subsubsection*{Case 3: $f(v_{32}) = 1$}

\begin{itemize}

    \item $f(v_{23}) = 3$ or $f(v_{23}) = 0$: can be treated similarly as the previous cases $f(v_{32}) = 3, f(v_{23}) = 1$, and $f(v_{32}) = 0, f(v_{23}) = 1$.

    \item $f(v_{23}) = 6 \Rightarrow \{v_{22},v_{33}\} \in \{3,4\}!$.
    
    \item $f(v_{23}) = 5 \Rightarrow \{v_{22},v_{33}\} \in \{3\}!$.
    
    \item $f(v_{23}) = 4 \Rightarrow \{v_{22},v_{33}\} \in \{6\}!$.
    
    \item $f(v_{23}) = 2 \Rightarrow f(v_{22}) = 4$ and $f(v_{33}) = 6$, or $f(v_{22}) = 6$ and $f(v_{33}) = 4$. In the first case, $f(v_{43}) = 3$ and $f(v_{34}) = 0 \Rightarrow v_{44}!$; in the second, $f(v_{21}) = 3$ and $f(v_{12}) = 0 \Rightarrow v_{11}!$.
    
\end{itemize}

\subsubsection*{Case 4: $f(v_{32}) = 2$}

\begin{itemize}

    \item $f(v_{23}) \in \{0,1,3\}$: can be treated similarly as in previous cases with $f(v_{23})$ and $f(v_{32})$ swapped.

    \item $f(v_{23}) = 5 \Rightarrow \{v_{22},v_{33}\} \in \{0\}!$.
    
    \item $f(v_{23}) = 6 \Rightarrow f(v_{22}) = 0$ and $f(v_{33}) = 4$, or $f(v_{22}) = 4$ and $f(v_{33}) = 0$. In the first case, $v_{43}!$; in the second, $v_{21}!$.
    
    \item $f(v_{23}) = 4 \Rightarrow f(v_{22}) = 0$ and $f(v_{33}) = 6$, or $f(v_{22}) = 6$ and $f(v_{33}) = 0$. In the first case,  $v_{43}!$; in the second, $v_{21}!$.
    
\end{itemize}

\end{proof}

\begin{figure}
    \centering
    
    \includegraphics[width=6.75cm, height=6cm]{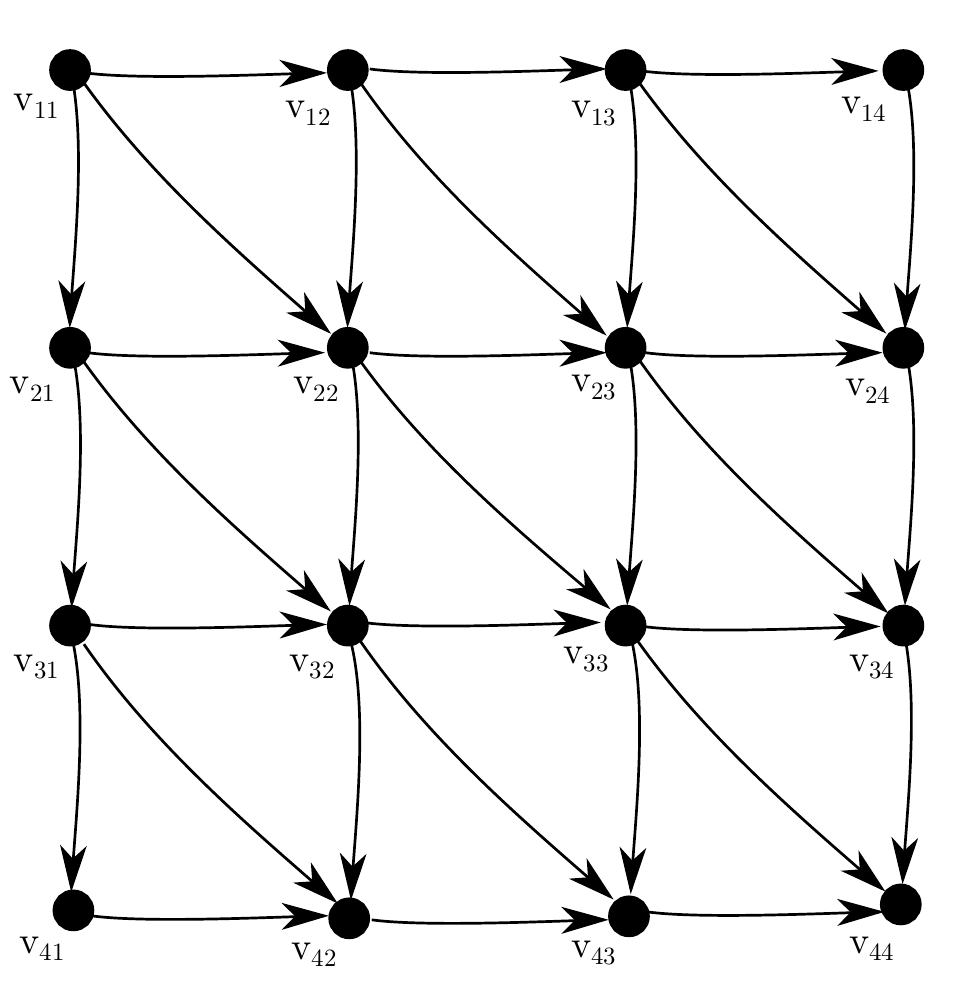}
    
    \caption{A $\protect\overrightarrow{P_4} \boxtimes \protect\overrightarrow{P_4}$ subgraph of $\protect\overrightarrow{C_m} \boxtimes \protect\overrightarrow{C_n}$}
    \label{figstrong}
    
\end{figure}

\begin{proof}[Proof of Theorem \ref{strong}]

By Theorem \ref{strongold}, it is enough to prove that $G = \lambda(\overrightarrow{C_m} \boxtimes \overrightarrow{C_n}) = 6$ if and only if 7 divides $m$ and $n$. 

\medskip

If both $m$ and $n$ are divisible by 7, the following periodic labeling is easily checked to be a $L(2,1)$-labeling of $\overrightarrow{C_m} \boxtimes \overrightarrow{C_n}$: the pattern 0246135 is repeated along the cycles. More explicitly, $f(i,j) = 0, 2, 4, 6, 1, 3, 5$ if $i + j \equiv 2, 3, 4, 5, 6, 0, 1 \pmod 7$, respectively.

\medskip

On the other hand, assume that $G$ admits a 6-$L(2,1)$-labeling $f$. By Lemma \ref{strongperiod}, $f$ is diagonal. Similarly as in the proof of Theorem \ref{large}, it is simple to check that, if $m \geq n + 3$, then $f$ induces a 6-$L(2,1)$-labeling of $\overrightarrow{C_m}_{-n} \boxtimes \overrightarrow{C_n}$, simply by considering the restriction of the coloring in a $\overrightarrow{C_m}_{-n} \boxtimes \overrightarrow{C_n}$ subgraph of $\overrightarrow{C_m} \boxtimes \overrightarrow{C_n}$.  Again, applying this argument consecutively, we are either left with a $\overrightarrow{C_d} \boxtimes \overrightarrow{C_d}$, where $d = \gcd(m,n)$, if $d \geq 3$, or a $\overrightarrow{C_{k+1}} \boxtimes \overrightarrow{C_k}$, or a $\overrightarrow{C_{k+2}} \boxtimes \overrightarrow{C_k}$.

\medskip

In the last two cases, the fact that $f$ is diagonal immediately implies that there are either two consecutive vertices or two vertices within distance two which receive the same color, which is a contradiction.

\medskip

We are done, then, if we prove that $\lambda({\overrightarrow{C_d} \boxtimes \overrightarrow{C_d}}) \geq 7$ if $d$ is not a multiple of 7. Indeed, assume that there is a 6-$L(2,1)$-labeling of $\overrightarrow{C_d} \boxtimes \overrightarrow{C_d}$. Again, by Lemma \ref{strongperiod}, it should be diagonal. In particular, it means that it corresponds to a labeling of the cycle $C_d$ with the following property: every pair of vertices with distance at most two must receive colors two apart, and every pair of vertices with distance three or four must receive distinct colors: indeed, the vertex $(i+1,j)$ is a neighbor of $(i,j)$, and the vertex $(i+2,j)$ has the same color as $(i+1,j+1)$, which is adjacent to $(i,j)$ in $\overrightarrow{C_d} \boxtimes \overrightarrow{C_d}$, so they must receive colors two apart from the color of $(i,j)$; similarly we prove that $(i+3,j)$ and $(i+4,j)$ must receive distinct colors from $(i,j)$. Such a coloring is denoted in the literature by $L(2,2,1,1)$-labeling. Note that, in particular, this implies the statement for $d = 3$ and $d = 4$. Let us assume in what follows that $d \geq 5$.

\medskip

If $C_d$ has such a coloring, it is readily checked that it must use the color 0 or 6. Indeed, otherwise all available colors are 12345, and it is impossible to color a $P_5$ subgraph of $C_d$ with only these colors.

\medskip

By symmetry, we may assume that 0 is used. Let $c$ be the coloring of $C_d$, and let the integers modulo $d$ represent its vertices. Let us assume that $c(0) = 0$. 

\medskip

If $c(1) = 3$, then 2 must receive a color from $\{5,6\}$. If $c(2) = 5$, then $c(3) = 1$, and then $c(-1) = 6$, $c(-2) = 2$ or $c(-2)=4$. In the first case, $c(-3) = 4$, and finally there is no available color for $-4$; in the second, $c(-3) \in \{1,2\}$, and in either case there is no color available for -4.
If $c(2) = 6$, then $c(3) = 1$, and then $c(-1) = 5$, $c(-2) = 2$, and there is no available color for -3.

\medskip

If $c(1) = 4$, then $c(2)$ is either 2 or 6. In the first case, $c(-1) = 6$ and there is no available color for 3. In the second, $c(-1) = 2$, which implies $c(-2) = 5$ and then there is no available color for -3.

\medskip

If $c(1) = 6$, then $c(2) \in \{2,3,4\}$. The first implies $c(3) = 4$ and there is no color for 4. The second implies $c(3) = 1$, and then $c(4) = 5$ and there is no color for 5. Finally, the third implies that either $c(3) = 1$ or $c(3) = 2$, both of which makes impossible to find a color for 4.

\medskip

The argument above implies that the neighbors of 0 must have colors 2 and 5. Without loss of generality, we may assume that $c(1) = 5$ and $c(-1) = 2$. This implies that $c(2) = 3$, which in turn implies $c(3) = 1$, then $c(4) = 6$ and $c(5) = 4$. It follows that $c(6) \in \{0,2\}$, but by the paragraphs above, $4$ cannot be a neighbor of $0$, so $c(6) = 2$. Then $c(7) = 0$, and the block 2053164 of size 7 is repeated. The only way the coloring can be completed along the cycle is, then, if 7 divides $d$.

\end{proof}

\section{Final remarks}

The natural next step would be to close the gap left from Theorem \ref{strong}, deciding for which $m$ and $n$ we have $\lambda({\overrightarrow{C_m} \boxtimes \overrightarrow{C_n}}) = 7$. 

\medskip

In the proof of Theorem \ref{strong}, we gave a periodic 6-labeling of $\lambda({\overrightarrow{C_7} \boxtimes \overrightarrow{C_7}})$, namely that one in which the pattern 0246135 is repeated along the cycles diagonally. In a similar fashion, the following periodic 7-coloring works for $\lambda({\overrightarrow{C_8} \boxtimes \overrightarrow{C_8}})$: 02461357. Concatenating these two patterns, one can show that $\lambda({\overrightarrow{C_m} \boxtimes \overrightarrow{C_m}}) = 7$ for every sufficiently large $m$ (namely, for every $m \in S(7,8)$; in particular for $m \geq 42$), and consequently $\lambda({\overrightarrow{C_m} \boxtimes \overrightarrow{C_n}}) = 7$ for every $m$, $n$ such that 7 does not divide both $m$ and $n$ and $\gcd(m,n) \geq 42$.

\medskip

Finally, we remark that it is simple to check that the proof of Lemma \ref{multiple} works in the setting of strong product of cycles as well. As we know from the paragraph above that $\lambda({\overrightarrow{C_m} \boxtimes \overrightarrow{C_m}}) = 7$ for every $m$ in $S(7,8)$ (and, in particular
for every $m \geq 42$), to prove that $\lambda({\overrightarrow{C_m} \boxtimes \overrightarrow{C_n}}) = 7$ for all sufficiently large $m$ and $n$, it would be enough to find a pair of coprime integers $a$, $b \in S(7,8)$ such that $\lambda({\overrightarrow{C_a} \boxtimes \overrightarrow{C_b}}) = 7$.

\bibliographystyle{acm}
\bibliography{ref}

\begin{thebibliography}{1}

\bibitem{calamoneri2011h}
{\sc Calamoneri, T.}
\newblock The ${L}(h,k)$-labelling problem: an updated survey and annotated
  bibliography (2014).
\newblock Available on
  \url{http://www.dsi.uniroma1.it/~calamo/PDF-FILES/survey.pdf}.

\bibitem{chang20032}
{\sc Chang, G.~J., and Liaw, S.-C.}
\newblock The ${L}(2, 1)$-labeling problem on ditrees.
\newblock {\em Ars Combinatoria 66\/} (2003), 23--31.

\bibitem{hale1980frequency}
{\sc Hale, W.~K.}
\newblock Frequency assignment: Theory and applications.
\newblock {\em Proceedings of the IEEE 68}, 12 (1980), 1497--1514.

\bibitem{shao20182}
{\sc Shao, Z., Jiang, H., and Vesel, A.}
\newblock ${L}(2, 1)$-labeling of the cartesian and strong product of two
  directed cycles.
\newblock {\em Mathematical Foundations of Computing 1}, 1 (2018), 49--61.

\bibitem{sylvester1884mathematical}
{\sc Sylvester, J.~J., et~al.}
\newblock Mathematical questions with their solutions.
\newblock {\em Educational times 41}, 21 (1884), 6.

\bibitem{yeh1990labeling}
{\sc Yeh, K.-C.}
\newblock {\em Labeling graphs with a condition at distance two}.
\newblock PhD thesis, University of South Carolina, 1990.

\bibitem{yeh2006survey}
{\sc Yeh, R.~K.}
\newblock A survey on labeling graphs with a condition at distance two.
\newblock {\em Discrete Mathematics 306}, 12 (2006), 1217--1231.

\end{thebibliography}

\end{document}